\documentclass[reqno]{amsart}

\usepackage{amsmath,amssymb,amsthm}
\usepackage{graphicx,tikz}
\newtheorem{theorem}{Theorem}

\newtheorem*{proposition}{Proposition}

\newtheorem*{definition}{Definition}

\theoremstyle{remark}

\begin{document}

\title[]{Buffon Discrepancy and the Steinhaus Longimeter}
 \thanks{}

\author[]{Stefan Steinerberger}
\address{Department of Mathematics and Department of Applied Mathematics, University of Washington, Seattle, WA 98195, USA}
\email{steinerb@uw.edu}

\begin{abstract} Let $\Omega \subset \mathbb{R}^2$ be a convex set. We study the problem of distributing a one-dimensional set $S$ with total length $L$ so that for any line $\ell$ in $\mathbb{R}^2$ the number of intersections $\#(\ell \cap S)$ is proportional to the length $\mathcal{H}^1(\ell \cap \Omega)$ as much as possible; we use the term Buffon discrepancy for the largest error. A construction of Steinhaus can be generalized to prove the existence of sets with Buffon discrepancy $\lesssim L^{1/3}$. We also show that the unit disk $\mathbb{D}$ admits a set with uniformly bounded Buffon discrepancy as $L \rightarrow \infty$.

\end{abstract}

\maketitle

\section{Introduction}
\subsection{The problem}
The goal of this paper is to present an interesting problem: we are given a bounded, convex domain $\Omega \subset \mathbb{R}^2$.  The goal is to find, for any given $L > 0$, a one-dimensional set $S \subset \Omega$ with length $L$ so that the following is true: for any line $\ell$ in $\mathbb{R}^2$ the number of intersections  $\# (\ell \cap S)$ is proportional (as much as possible) to the length of the line inside $\Omega$, that is $\mathcal{H}^1(\ell \cap \Omega)$. The obvious questions are: how can this be made precise, how well can it be done and what do extremal sets $S$ look like? Computational experiments (see Fig. \ref{fig:1}) suggest that the answer to these questions might be quite interesting; the question is also of obvious interest in higher dimensions; this paper is solely focused on the case of two dimensions.

\begin{center}
\begin{figure}[h!]
\begin{tikzpicture}
\node at (0,0) {\includegraphics[width=0.4\textwidth]{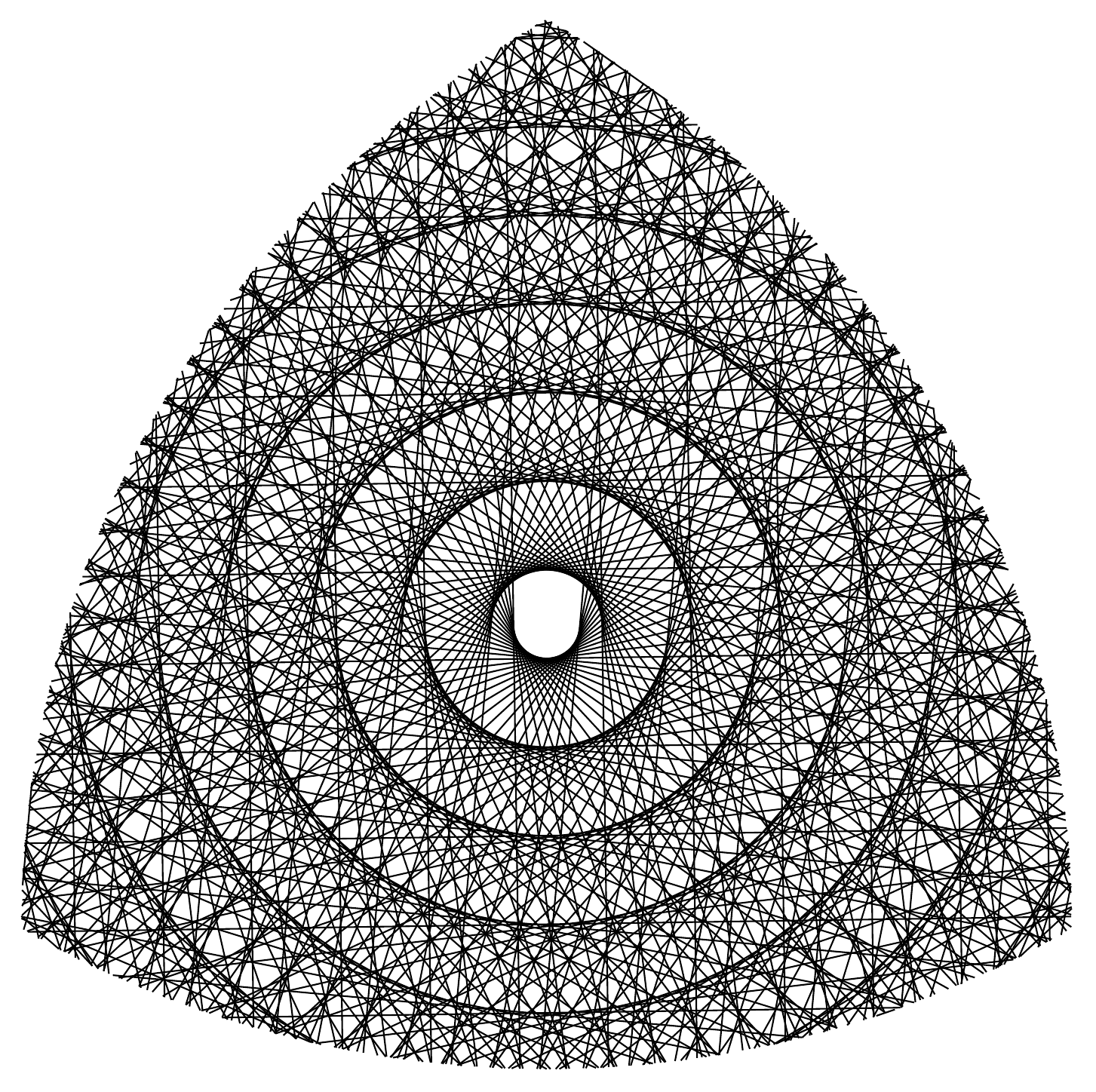}};
\node at (-6.5,0) {\includegraphics[width=0.4\textwidth]{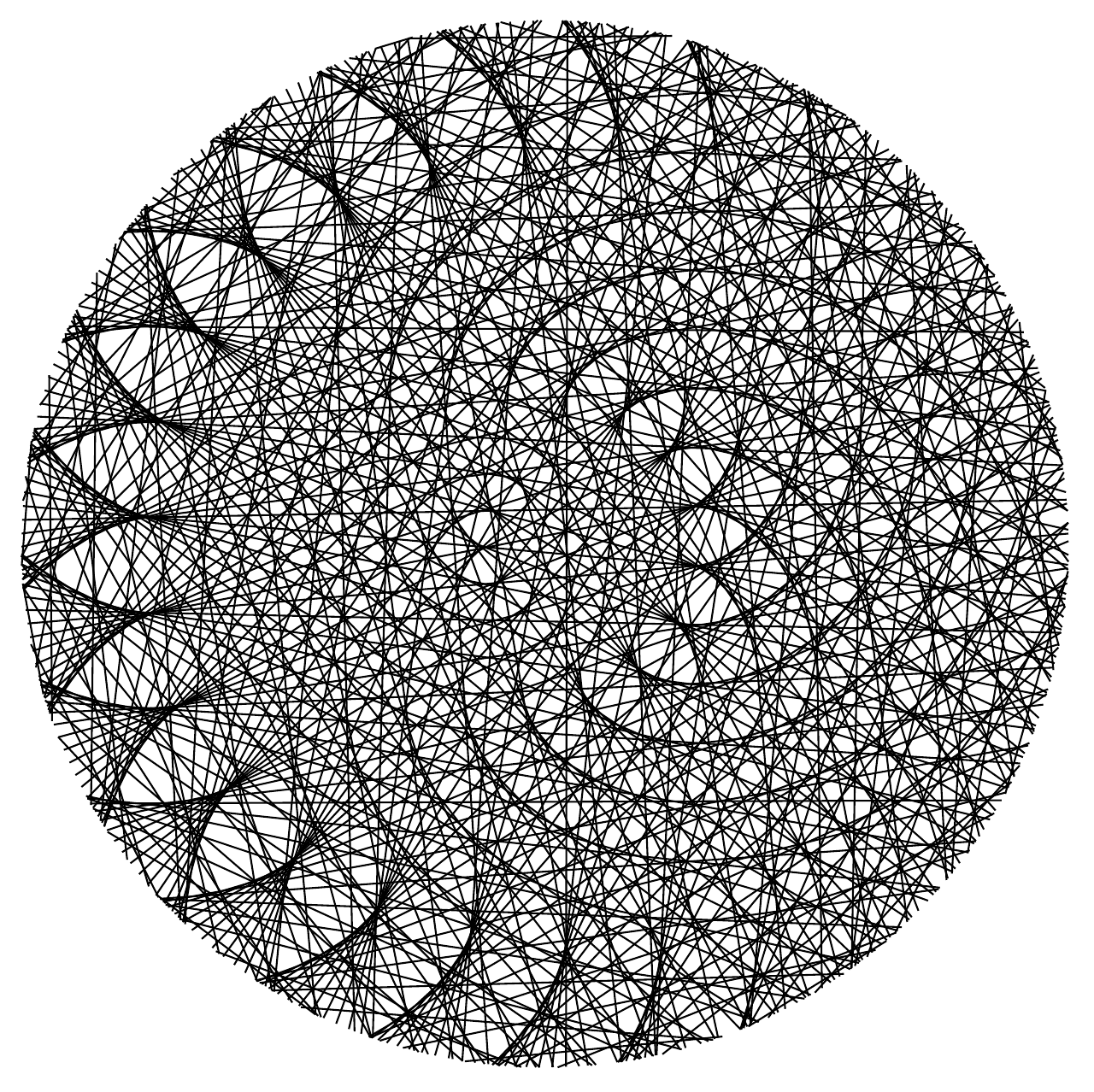}};
\end{tikzpicture}
\caption{Sets of lines of length $L =500$ inside the unit disk (left) and the Reuleaux triangle (right) with the property that every line through the domain intersects them a number of times roughly proportional to the length (with a fairly small error).}
\label{fig:1}
\end{figure}
\end{center}

\subsection{Buffon discrepancy} We start by motivating the quantity of interest. Asking for the number of intersections to be \textit{proportional} introduces two variables: (a) the proportionality factor and (b) the worst case error assuming that proportionality factor. It is clear that these two numbers (denoted $c$ and $X$ in the subsequent Proposition) are highly connected and maybe one does not wish to deal with both at the same time. Our first observation is that there exists a somewhat canonical proportionality factor $c$ suggested by the Cauchy-Crofton formula.  

\begin{proposition}[Cauchy-Crofton scaling]
Let $\Omega \subset \mathbb{R}^2$ be bounded, let $S$ be a set with length $\mathcal{H}^1(S) = L$ having the property that for some $c>0$ and all lines $\ell$ in $\mathbb{R}^2$
$$ \left| \# (\ell \cap S) - c\cdot \mathcal{H}^1( \ell \cap \Omega) \right| \leq X,$$
then
$$ \left| c - \frac{2}{\pi}\frac{ L}{\emph{area}(\Omega)} \right| \leq \frac{2 \cdot \emph{diam}(\Omega) }{\emph{area}(\Omega)} X.$$
\end{proposition}

Since we are mainly interested in the asymptotic regime $L \gg X$, the Proposition suggests to \textit{prescribe} the proportionality factor as $2 L/(\pi \cdot \mbox{area}(\Omega))$. We use this as our starting point to define Buffon discrepancy. 

\begin{definition} Let $\Omega \subset \mathbb{R}^2$ be a bounded set in the plane and let $S \subset \Omega$ be a rectifiable set with length $L$. We define the \emph{Buffon discrepancy} of $S$ (with respect to $\Omega$) as
$$ \left\|  \# (\ell \cap S) - \frac{2}{\pi}\frac{L}{ \emph{area}(\Omega)} \mathcal{H}^1(\ell \cap \Omega) \right\|_{L^{\infty}(\mu)},$$
where $\mu$ is the kinematic measure and $\ell$ ranges over all lines.
\end{definition}
Note that if the set $S$ were to contain a line segment, then there exists a line $\ell$ such that $\# (\ell \cap S) = \infty$. One could replace line segments by slightly curved segments and take a limit but presumably the more natural way out is to exclude `exceptional lines of measure 0' and this is done by the use of the kinematic measure; recall the formal definition of the $L^{\infty}-$norm as
 $$ \|f\|_{L^{\infty}(\mu)} = \inf\left\{C \geq 0: |f(x)| \leq C \quad \mbox{for}~\mu~\mbox{almost every}~x \right\}$$
 which encapsulates the idea that individual lines as well as sets of lines with measure 0 do not matter.

\begin{center}
\begin{figure}[h!]
\begin{tikzpicture}
\node at (0,0) {\includegraphics[width=0.35\textwidth]{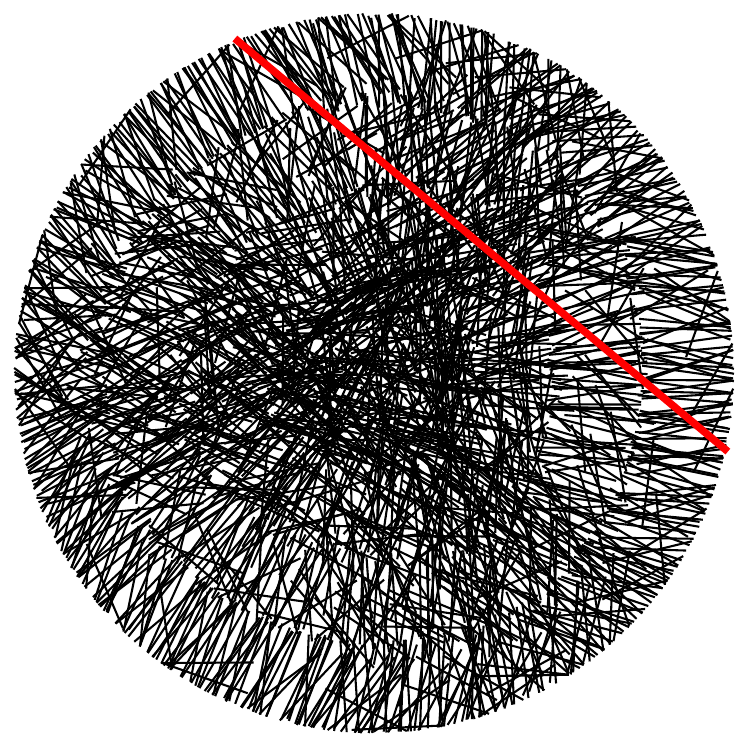}};
\node at (6,1) {2000 line segments $S$ with total length $L=500$};
\node at (6,0) {the red line $\ell$ intersects $\# (\ell \cap S) =160$};
\node at (6.1,-0.4) {segments and has length $\mathcal{H}^1(\ell) \sim 1.769$};
\node at (6, -1.5) {$ \Large \left| 160 - \frac{2L}{\pi^2} \cdot 1.769 \right| \sim 19.24$};
\end{tikzpicture}
\caption{A set with length $L=500$ in the unit disk and a line showing that the Buffon discrepancy of the set is $\geq 19.24$.}
\end{figure}
\end{center}

There exists a trivial lower bound: for fixed $\Omega$ as $L \rightarrow \infty$, we observe that the intersection $ \# (\ell \cap S)$ is always integer-valued while the length of all chords assume intermediate real values; in general, we always have a trivial bound
$$ \left\|  \# (\ell \cap S) - \frac{2}{\pi}\frac{L}{\mbox{area}(\Omega)} \mathcal{H}^1(\ell \cap \Omega) \right\|_{L^{\infty}(\mu)} \geq \frac{1}{2}.$$
As it turns out, this trivial lower bound is optimal in the case of the unit disk.

 \subsection{Unit Disk}
The unit disk $\mathbb{D}$ is a natural first example. We prove the existence of a very simple set, a union of suitably placed concentric circles, with uniformly bounded Buffon discrepancy.

\begin{theorem}
For any $L > 0$ there exists a set $S \subset \mathbb{D}$ with total length $L$ with
 $$ \max_{\ell} \left|  \# (\ell \cap S) - \frac{2L}{\pi^2} \mathcal{H}^1(\ell \cap \mathbb{D}) \right| \leq 100.$$
\end{theorem}

\begin{center}
\begin{figure}[h!]
\begin{tikzpicture}
\node at (0,0) {\includegraphics[width=0.4\textwidth]{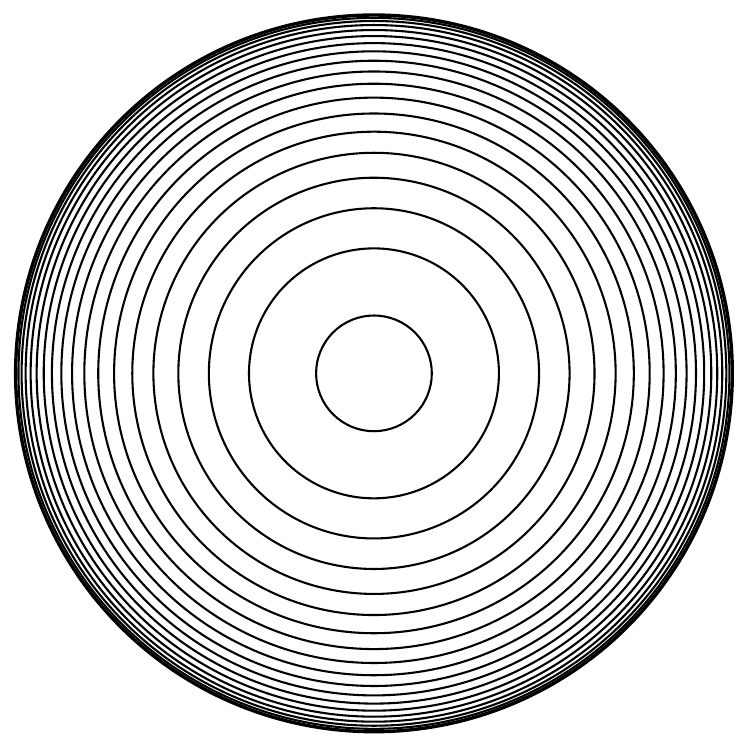}};
\node at (6,0) {\includegraphics[width=0.4\textwidth]{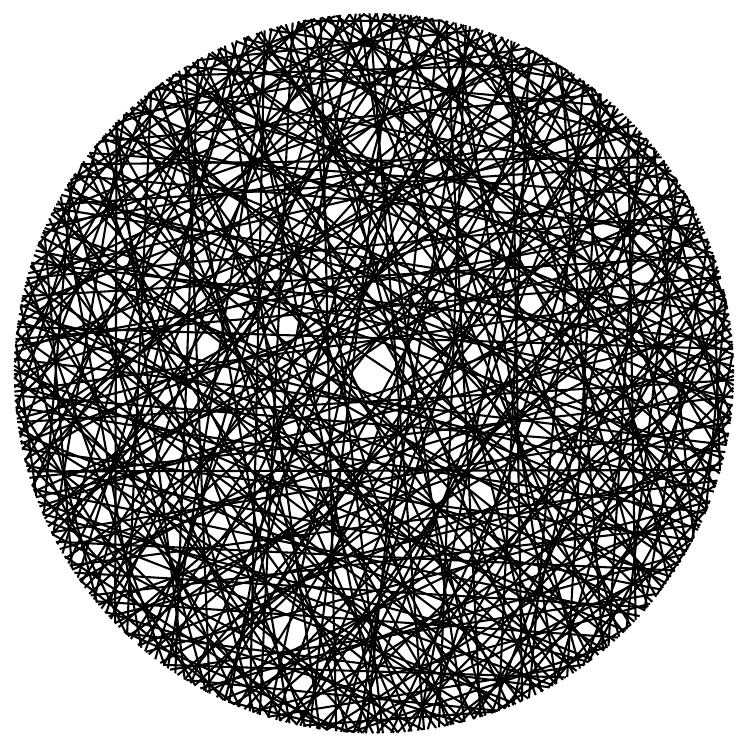}};
\end{tikzpicture}
\caption{Two sets of line segments with total length $L=500$ inside the unit disk with very small Buffon discrepancy.}
\label{fig:disk}
\end{figure}
\end{center}

It seems, see Fig. \ref{fig:1} and Fig. \ref{fig:disk}, that there are different types of sets inside the unit disk that have a small Buffon discrepancy. The union of concentric circles is arguably the simplest one; however, it might be interesting to see if it is possible to construct other examples of such sets.

\subsection{The Steinhaus longimeter} 
This type of problem appears to be very naturally related to a construction first proposed (and patented!) by Hugo 
Steinhaus \cite{stein00, stein0, steinhaus, stein2}. His 1930 paper \textit{On the practice of rectification and notion of length} \cite{stein00} begins by saying that

\begin{quote}
This note belongs to the area of applied mathematics. It proposes a method which allows the optical measurement of physically given curves. The term `optical' should highlight the distinction between our method and the mechanical [...] There are many conceivable cases where a mechanical device cannot be used. [...] for example if one wants to measure the length of a string-like curved object under the microscope [...] (Steinhaus, \cite{stein00})
\end{quote} 

One of the applications he has in mind is measuring the length of objects on a map as is also explained in his 1931 publication in \textit{Czasopismo Geograficzne} (Geographic Journal) \cite{stein0}.
Steinhaus first explains the usual proof of the Crofton formula and then notes that the proof allows for a discretization: taking 6 lines with an angle of 30 degrees between them and all their translates, Steinhaus notes that the number of intersections between this union of lines and any line segment is nearly proportional to the length of the line segment up to an error between $-2.26\%$ and $1.15\%$. More generally, let $S_{n, \varepsilon}$ denote $n$ lines going through the origin at an equal angle together with all translates by $\varepsilon$, 
$$ S_{n, \varepsilon} = \left\{ \begin{pmatrix} \cos\left( \pi k/n \right) \\ \sin\left(  \pi k/n\right) \end{pmatrix} t + \begin{pmatrix} - \sin\left( \pi k/n \right) \\ \cos\left(  \pi k/n\right) \end{pmatrix} s \varepsilon:  t \in \mathbb{R},~ s \in \mathbb{Z}, ~0 \leq k < n\right\}.$$
 The set $S_{1, \varepsilon}$ is the union of lines parallel to the $x-$axis while $S_{2, \varepsilon}$ is the familiar grid and $S_{3, \varepsilon}$ leads to a hexagonal structure. The standard (patented) Steinhaus longimeter construction corresponds to $S_{6, \varepsilon}$. 

\begin{center}
\begin{figure}[h!]
\begin{tikzpicture}
\node at (0,0) {\includegraphics[width=0.18\textwidth]{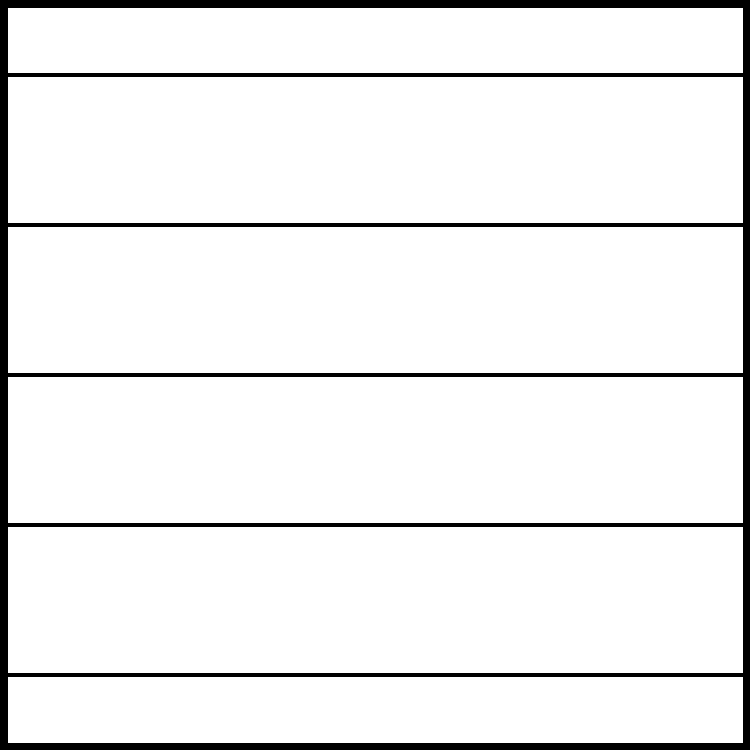}};
\node at (2.5,0) {\includegraphics[width=0.18\textwidth]{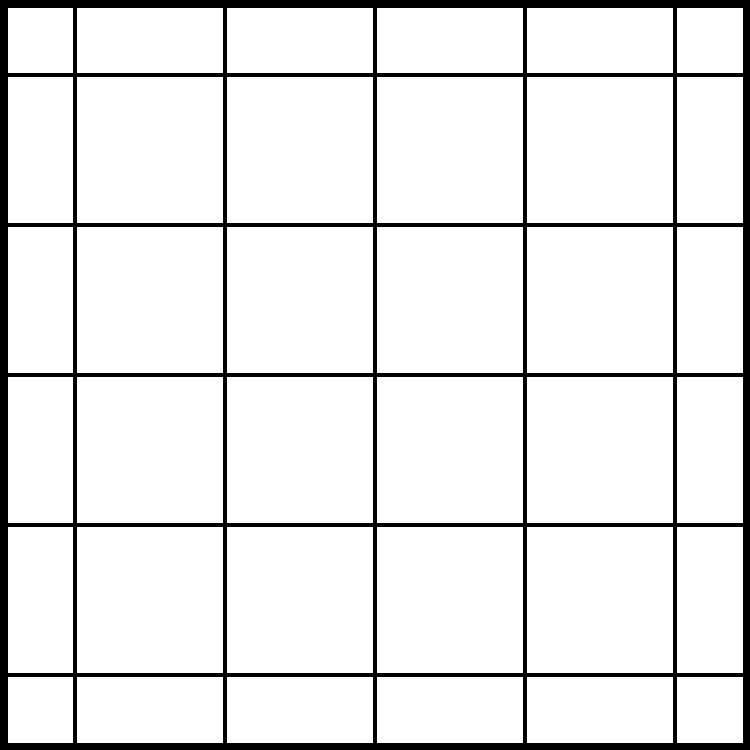}};
\node at (5,0) {\includegraphics[width=0.18\textwidth]{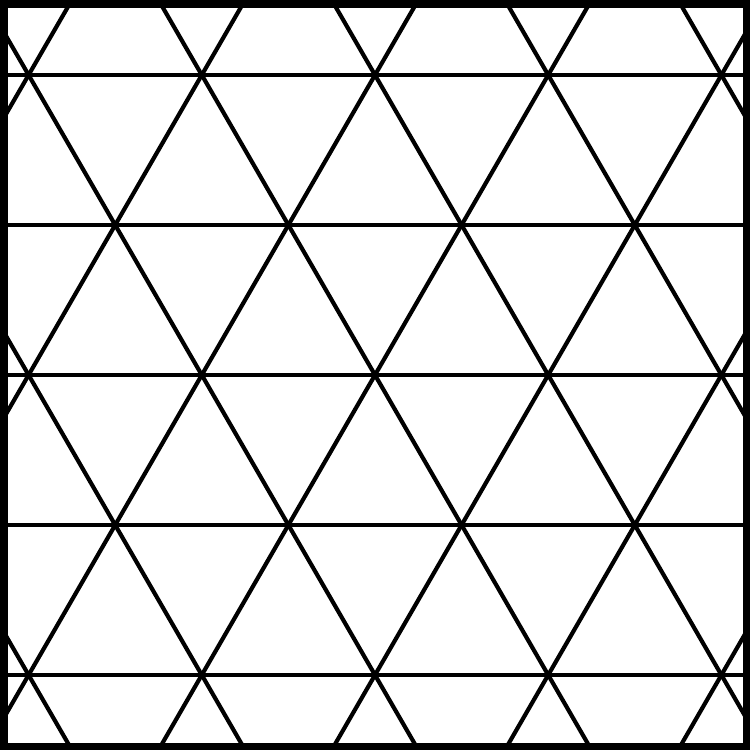}};
\node at (7.5,0) {\includegraphics[width=0.18\textwidth]{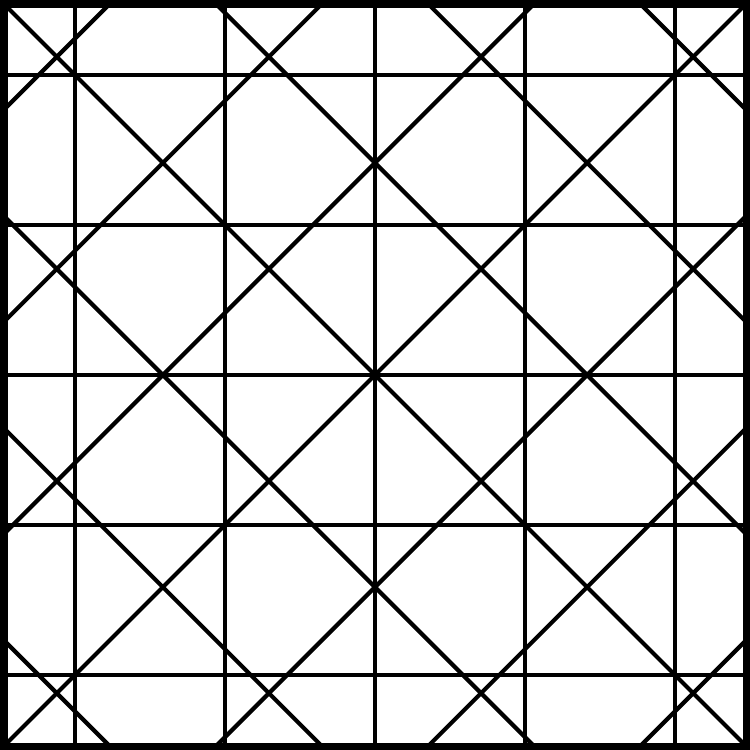}};
\node at (10,0) {\includegraphics[width=0.18\textwidth]{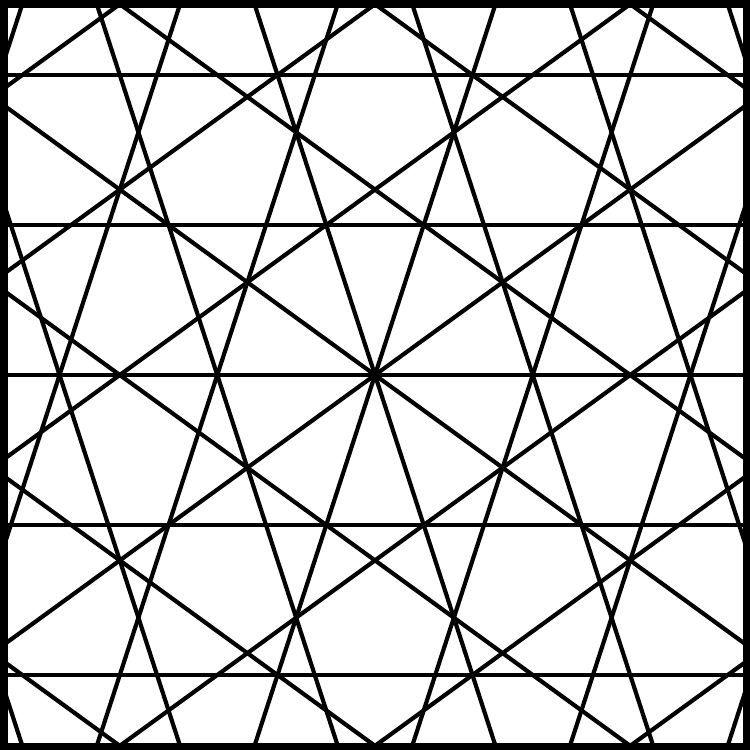}};
\end{tikzpicture}
\vspace{-10pt}
\caption{The sets $S_{n,1/5}$ for $1 \leq n \leq 5$.}
\label{fig:ssets}
\end{figure}
\end{center}
\vspace{-20pt}
We will now argue that this generalized Steinhaus longimeter provides a universal bound for our problem: as $L \rightarrow \infty$, the set $S_{n, \varepsilon} \cap \Omega$, for suitable $n, \varepsilon$, is a set with small Buffon discrepancy.

\begin{theorem} Let $\Omega \subset \mathbb{R}^2$ be a bounded, convex domain.  Then there exists $c_{\Omega}$ such that, as $L \rightarrow \infty$, the set
$ S = \Omega \cap S_{L^{1/3}, L^{-2/3}} $
has length $\mathcal{H}^1(S) \sim L$ and
$$ \left\|  \# (\ell \cap S) -  \frac{2}{\pi}\frac{L}{\emph{area}(\Omega)} \mathcal{H}^1(\ell \cap \Omega) \right\|_{L^{\infty}(\mu)}  \leq c_{\Omega} \cdot L^{1/3}.$$
\end{theorem}

 It is not difficult to see that this result is sharp: the set $\Omega$ could be oriented so that the origin $(0,0) \in \partial \Omega$.  However, the Steinhaus set $S_{n, \varepsilon}$ has the property that $n$ lines meet at the origin, therefore this bound is the best one can hope for when $n \sim L^{1/3}$. One might naturally wonder whether it is possible to avoid this worst case scenario by suitably moving and rotating the set $\Omega$, perhaps an additional averaging procedure can lead to a further improvement.

 \subsection{Open problems.} These two results, a uniformly bounded Buffon discrepancy for the unit disk and a bound of $L^{1/3}$ for general convex sets, naturally suggest a number of different questions. We only list some of the more obvious ones.
 
 \begin{enumerate}
 \item Is it always possible to find a set with Buffon discrepancy of order $\sim 1$ for any convex domain?  If so, is there a simple construction?  If not, what is the best Buffon discrepancy one can hope for?
 \item Is it possible to find other explicit examples with Buffon discrepancy of order $\sim 1$ for the unit disk?  Fig.\ref{fig:1} and Fig. \ref{fig:disk} show numerically obtained examples suggesting their existence. 
 \item What happens if we were to restrict ourselves to only working with sets $S$ that are a union of intersections of lines with $\Omega$? This would include all Steinhaus sets but would eliminate examples like Fig. \ref{fig:disk}. Would this fundamentally change the problem? 
 \item What about the analogous problem in higher dimensions? Note that, for example in $\mathbb{R}^3$, there are at least two separate problems depending on whether `lines' are understood to be lines or sets of co-dimension 1.
 \end{enumerate}
 
\begin{center}
\begin{figure}[h!]
\begin{tikzpicture}
\node at (4,-4) {\includegraphics[width=\textwidth]{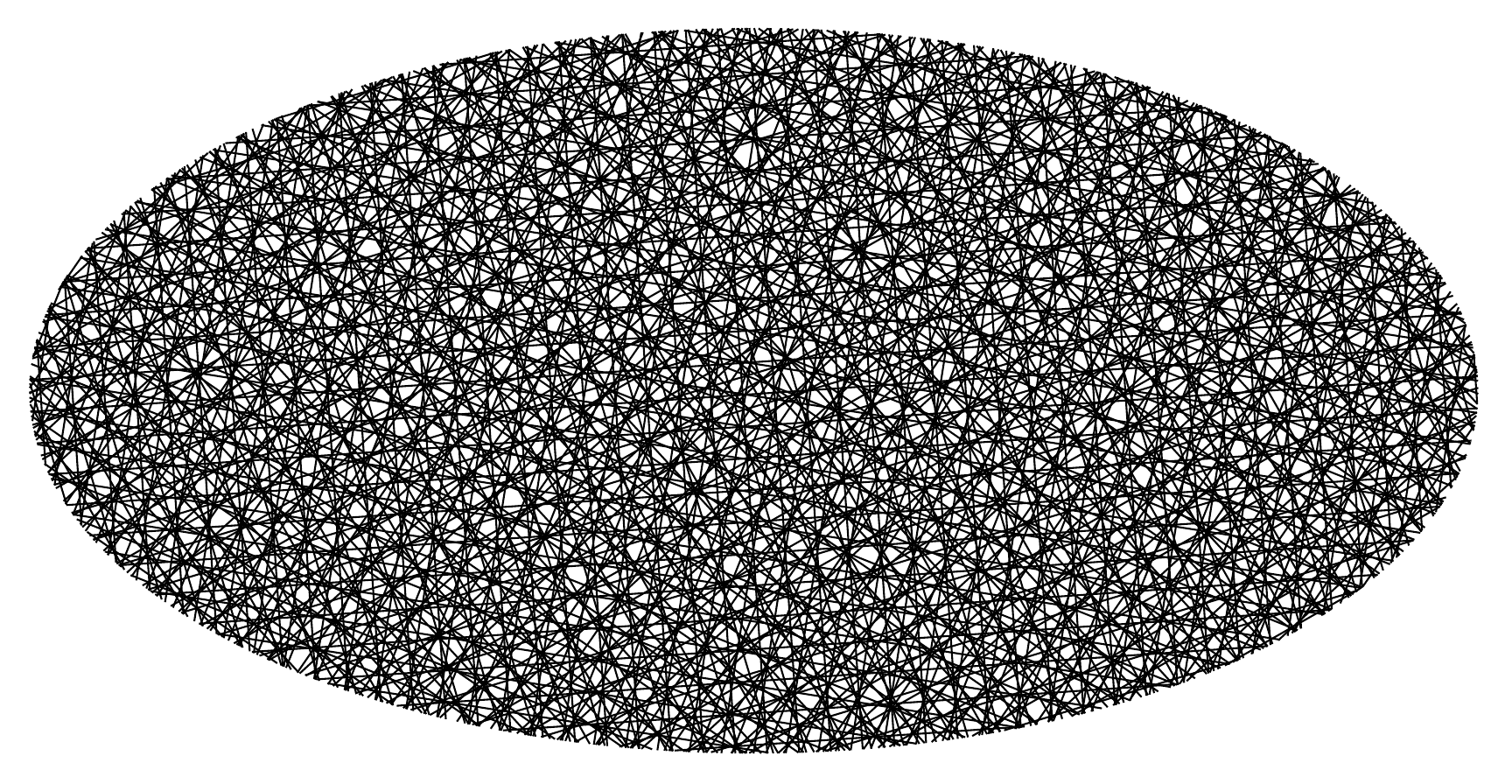}};
\end{tikzpicture}
\vspace{-25pt}
\caption{Construction in the style of the Steinhaus longimeter.}
\label{fig:ellipse}
\end{figure}
\end{center}
\vspace{-10pt}
It appears as if Steinhaus sets are a fairly natural starting point when it comes to the construction of optimal sets; there are several natural variations one might want to consider, for example, we defined them to all meet in the origin but one could add an offset to each individual line to avoid this cluster.  A second interesting aspect is that the Steinhaus sets $S_{n, \varepsilon}$, especially for $n$ large and $\varepsilon$ small, have a fairly intriguing micro-structure (see, Fig. \ref{fig:s50} and Fig. \ref{fig:s51}). There are regions that appear fairly random, there are obvious clusters where many lines meet and there are some relatively empty regions. Theorem 2 suggests that $S_{n, 1/n^2}$ for $n \in \mathbb{N}$ might be a particularly interesting subfamily.
\vspace{-8pt}
\begin{center}
\begin{figure}[h!]
\begin{tikzpicture}
\node at (4,-4) {\includegraphics[width=\textwidth]{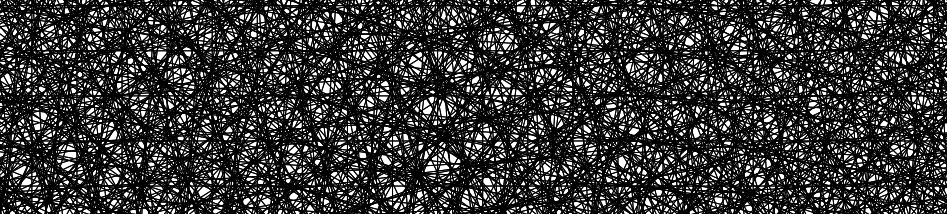}};
\end{tikzpicture}
\vspace{-20pt}
\caption{Microstructures in $S_{50, 1/2500}$.}
\label{fig:s50}
\end{figure}
\end{center}
\vspace{-5pt}

\subsection{Related results.} We are not aware of any directly related results in integral geometry \cite{santalo}.  Our problem can be understood in the spirit of discrepancy theory, see Beck-Chen \cite{beckchen}, Chazelle \cite{chaz}, Drmota-Tichy \cite{drmota}, Kuipers-Niederreiter \cite{kuipers} and Matousek \cite{matousek}. Indeed, it could be interpreted as the problem finding a suitable discretization of kinematic measure restricted to a domain, however, we know of no work in that direction. The Steinhaus longimeter has had profound impact in \textit{stereology}, the study of how to extract information from lower-dimensional sets, we refer to \cite{badd, ster00, coer, ster0, cruz, ster1, gomez, kid, kulk}. A related result is due to Liu-Zhang-Zheng-Paul \cite{liu}: they use a low-discrepancy sequence on the sphere to create `uniformly distributed' lines, though, in their setup the goal is surface estimation of an a priori unknown object. This has since been further pursued \cite{kot, neeman}, however, it is really a different problem. We also note the existence of the work of Ambartzumian \cite{amb1, amb2} which, again, seems to be of a very different flavor.

\section{Proofs}
\subsection{Proof of the Proposition.}
\begin{proof}
Since $\Omega$ is bounded, we may assume without loss of generality that it is contained in a ball of radius $R$ centered at the origin, i.e. $\Omega \subset B(0,R)$. We will now consider the set of all lines $\ell$ that intersect this ball which we can identify with the interval $[-R, R] \times \mathbb{S}^1$ equipped with the usual product measure: this identifies each line with the angle that it makes with the $x-$axis (or any other fixed reference direction) as well as the closest (signed) distance it has with the origin. By choosing the product measure, the arising measure is invariant under rotation and translation. This set has total measure $2R \times 2\pi = 4 R \pi$. Since the set $S$ is rectifiable, we may think of it as a union of finitely many line segments; we first argue that linearity gives that
$$ \int_{-R}^{R} \int_{0}^{2\pi} \# (\ell \cap S) dx d\theta = 4 L.$$
This can be seen as follows: if $S$ is a line segment of length $\varepsilon$, then the size of its projection depends on the angle $\alpha$ between the projection direction and the orientation of the line segment and evaluates to $|\cos{\alpha}| \varepsilon$. In that case, 
$$ \int_{-R}^{R} \int_{0}^{2\pi} \# (\ell \cap S) dx d\theta =  \int_{0}^{2\pi}  \int_{-R}^{R} \# (\ell \cap S) dx d\theta = \int_{0}^{2\pi} |\cos{\alpha}| \varepsilon d\alpha = 4 \varepsilon.$$
Since the integral is additive, we arrive at the desired conclusion.
Simultaneously, by Fubini, we have
\begin{align*} 
\int_{-R}^{R} \int_{0}^{2\pi} \mathcal{H}^1(\ell \cap \Omega) dx d\theta &= \int_{0}^{2\pi} \int_{-R}^{R} \mathcal{H}^1(\ell \cap \Omega) dx d\theta \\
&= \int_{0}^{2\pi} \mbox{area}(\Omega) d\theta = 2\pi \cdot  \mbox{area}(\Omega).
\end{align*}
If it is now true that for some $c, X > 0$ and all lines $\ell$ we have that
$$ \left| \# (\ell \cap S) - c \cdot \mathcal{H}^1( \ell \cap \Omega) \right| \leq X,$$
then the triangle inequality implies
\begin{align*}
\left| 4L - c \cdot 2\pi \cdot \mbox{area}(\Omega) \right| &= \left| \int_{-R}^{R} \int_{0}^{2\pi} \# (\ell \cap S)  - c \cdot \mathcal{H}^1( \ell \cap \Omega)  dx d\theta \right| \\
&\leq \int_{-R}^{R} \int_{0}^{2\pi}   \left|  \# (\ell \cap S)  - c \cdot \mathcal{H}^1( \ell \cap \Omega)  \right| dx d\theta \leq 4 \pi R X.
\end{align*}
Dividing both sides by $2 \pi \mbox{area}(\Omega)$ now leads to
$$ \left| c - \frac{2}{\pi}\frac{L}{\mbox{area}(\Omega)} \right| \leq \frac{2 R }{\mbox{area}(\Omega)} X.$$
Using $R \leq \mbox{diam}(\Omega)$ implies the result.
\end{proof}

\begin{center}
\begin{figure}[h!]
\begin{tikzpicture}
\node at (0,0) {\includegraphics[width=0.2\textwidth]{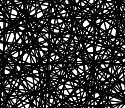}};
\node at (3,0) {\includegraphics[width=0.21\textwidth]{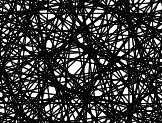}};
\node at (6,0) {\includegraphics[width=0.2\textwidth]{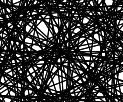}};
\node at (0,-2.4) {\includegraphics[width=0.2\textwidth]{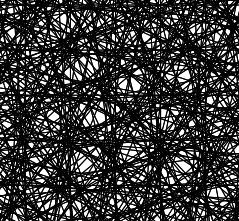}};
\node at (4.46,-2.4) {\includegraphics[width=0.445\textwidth]{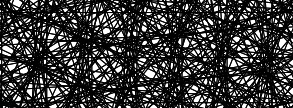}};
\end{tikzpicture}
\vspace{-10pt}
\caption{Pictures at an exhibition: $S_{50, 1/2500}$ at small scale.}
\label{fig:s51}
\end{figure}
\end{center}

\subsection{Proof of Theorem 1}
\begin{proof} 
We consider a union of circles centered at the origin.  More precisely, given radii $0 < r_1 < r_2 < \dots < r_k \leq 1$, we consider the set
$$ S = \bigcup_{i=1}^{k} \left\{x \in \mathbb{R}^2: \|x\| = r_i \right\}.$$
The length of this set is supposed to be $L$, this requires
$$ r_1 + r_2 + \dots + r_k = \frac{L}{2\pi}.$$
For any arbitrary line $\ell$, we consider its closest approach to the origin, i.e. $d(\ell) = \min_{x \in \ell} \|x\|$. Then
$$ \mathcal{H}^1(\ell \cap \mathbb{D}) = 2 \sqrt{1 - d(\ell)^2}.$$
Simultaneously, we have that, except for a set of kinematic measure 0
$$ \# (S \cap \ell) =  2 \cdot \# \left\{1 \leq i \leq k: d(\ell) \leq r_i \right\} .$$
The exceptional set of lines is exactly the set of lines that is tangent to one of the $k$ circles and this tells us that for \textit{all} lines
$$ \left| \# (S \cap \ell) -  2 \cdot \# \left\{1 \leq i \leq k: d(\ell) \leq r_i \right\} \right| \leq 1.$$
 We arrive at the bound
\begin{align*}
\left| \# (\ell \cap S) - \frac{2L}{\pi^2} \mathcal{H}^1(\ell \cap \mathbb{D}) \right| &\leq 1 + \left|  2 \cdot \# \left\{1 \leq i \leq k: d(\ell) \leq r_i \right\}   - \frac{2L}{\pi^2} \mathcal{H}^1(\ell \cap \mathbb{D}) \right| \\
&\leq 1 + 2 \left|  \# \left\{1 \leq i \leq k: d(\ell) \leq r_i \right\}   - \frac{2L}{\pi^2} \sqrt{1 - d(\ell)^2}\right|.
\end{align*}
The line $\ell$ has been reduced to a number $0 < d(\ell) < 1$ and it remains to find a suitable $k \in \mathbb{N}$ and a suitable set of $k$ radii $r_1 < r_2 < \dots < r_k$ subject
to the boundary conditions indicated above such that
$$ \max_{0 < r < 1} \left|  \# \left\{1 \leq i \leq k:  r_i \geq r \right\}   - \frac{2L}{\pi^2} \sqrt{1 - r^2}\right| \qquad \mbox{is small}.$$
This is merely the problem of discretizing a non-uniform density. More precisely, since $\sqrt{1-r^2}$ vanishes when $r=1$, there is no point in having anything at radius exactly 1. If one imagines starting at $r=1$ and slowly decreasing the value of $r$, we see that $(2L/\pi^2) \sqrt{1-r^2}$ is monotonically increasing. If one wanted to keep the error uniformly bounded, it might make sense to place a radius every time the value attains a new integer value, i.e. solving for
$$ \frac{2L}{\pi^2} \sqrt{1-r_i^2} = i$$
which suggests setting
$$ r_i = \sqrt{1 - \frac{i^2 \pi^4}{4 L^2}} \qquad \mbox{for} \qquad 1 \leq i < \frac{2L}{\pi^2}.$$
Then
$$  \sqrt{1 - \frac{i^2 \pi^4}{ 4L^2}} \geq r \qquad \mbox{is equivalent to} \qquad  \frac{2L}{\pi^2} \sqrt{1 - r^2} \geq i$$
which shows that the maximal error term is bounded from above by 1, by construction.  Of course, in doing so, there is no a priori guarantee that the length of the set is exactly $L$ and there is no reason why it would be. The total length is 
\begin{align*}
\mathcal{H}^1(S) = 2\pi  \sum_{i=1}^{\left\lfloor  2L /\pi^2 \right\rfloor}  \sqrt{1 - \frac{i^2 \pi^4}{4L^2}}.
\end{align*}
We approximate the sum by an integral (which leads to a controlled error because the integrand is in $[0,1]$ and monotonically decreasing) and arrive at the inequality
$$ \left| \mathcal{H}^1(S) -  2\pi \int_{0}^{2L/\pi^2} \sqrt{1 - \frac{\pi^4 }{4L^2} x^2} dx \right| \leq 8\pi.$$
The area of a quarter disk is merely
$$  \int_{0}^{2L/\pi^2} \sqrt{1 - \frac{\pi^4 }{4L^2} x^2} dx = \frac{L}{2\pi}$$
and we arrive at
$$ \left| \mathcal{H}^1(S) -  L \right| \leq 8\pi.$$
We can now simply add or remove the remaining length by adding/removing at most $\leq 20$ circles of radius 1 leading to a uniformly bounded error.
\end{proof}

\subsection{Proof of Theorem 2}
\begin{proof} Let $x,y \in \mathbb{R}^2$ be two distinct points in the plane.
 It is our goal to count the number of intersections between the line segment from $x$ to $y$ and the family of lines that belong to the Steinhaus set $S_{n, \varepsilon}$ which is given by
$$ \begin{pmatrix} \cos\left( \pi k/n \right) \\ \sin\left(  \pi k/n\right) \end{pmatrix} t + \begin{pmatrix} - \sin\left( \pi k/n \right) \\ \cos\left(  \pi k/n\right) \end{pmatrix} s \varepsilon  \qquad  \quad t \in \mathbb{R},~ s \in \mathbb{Z}, ~0 \leq k < n.$$
Our goal is to obtain uniform estimates which implies that the precise structure of the set $\Omega$ does not play a role: convexity enters implicitly in the sense that any line going through $\Omega$ will enter at a point $x \in \partial \Omega$ and leave the set $\Omega$ in another point $y \in \partial \Omega$ except for tangent lines in which case $x=y$ which is also covered by the subsequent argument.
\vspace{-10pt}
\begin{center}
\begin{figure}[h!]
\begin{tikzpicture}
\node at (0,0) {\includegraphics[width=0.4\textwidth]{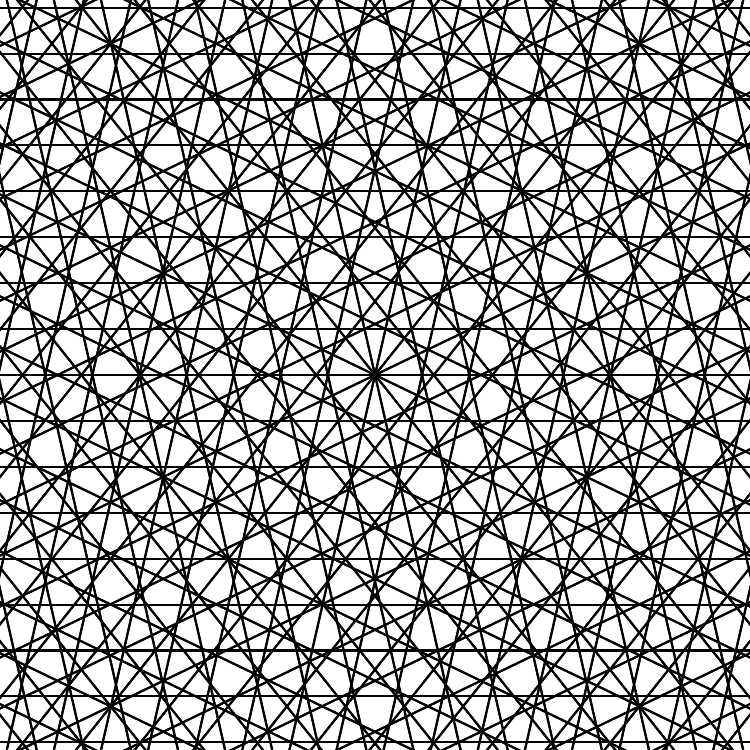}};
\filldraw[red] (-0.9,-1.2) circle (0.1cm);
\filldraw[red] (0.7,1.6) circle (0.1cm);
\draw [very thick, red] (-0.9, -1.2) -- (0.7, 1.6);
\filldraw[red] (-0.9,-1.2) circle (0.1cm);
\filldraw[red] (0.7,1.6) circle (0.1cm);
\filldraw[black] (7-0.9,-1.2) circle (0.1cm);
\filldraw[black] (7.7,1.6) circle (0.1cm);
\draw [very thick, black] (7-0.9, -1.2) -- (7+0.7, 1.6);
\foreach \i in {-16, -15, ..., 17} {
    \draw (5, 0.4 + \i*0.1) -- (8, \i*0.1);
}
\end{tikzpicture}
\vspace{-5pt}
\caption{Counting the number of intersections for a full Steinhaus set (left) and for the translates of a fixed line (right).}
\label{fig:break}
\end{figure}
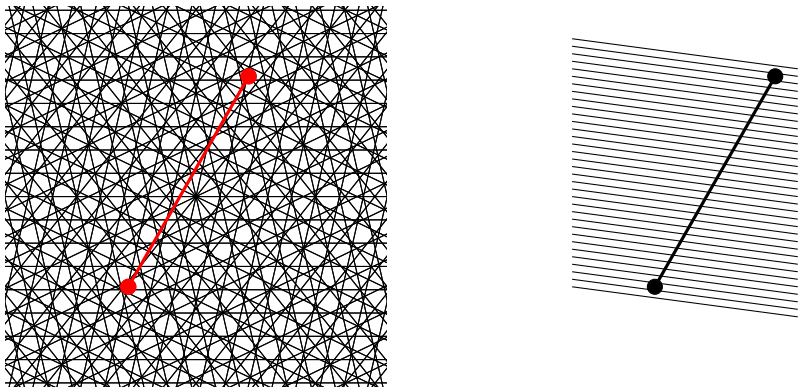
\end{center}
\vspace{-10pt}
We proceed by fixing $0 \leq k < n$ and counting the number of lines with fixed $k$ intersecting the line segment from $x$ to $y$.
A quick inspection, see Fig. \ref{fig:break}, shows that this depends on the projection of the line segment onto the orthogonal direction (up to an error of $\mathcal{O}(1)$ coming from the two endpoints); that number is
$$ \frac{1}{\varepsilon}\left| \left\langle \begin{pmatrix} - \sin\left( \pi k/n \right) \\ \cos\left(  \pi k/n\right) \end{pmatrix}, y - x \right\rangle \right| + \mathcal{O}(1).$$
Summing over $0 \leq k < n$  and writing $y-x = (-a,b)$ for some $a,b \in \mathbb{R}$, we see that the number of intersections is given by
 $$ \# \mbox{intersections} = \frac{1}{\varepsilon}\sum_{k=0}^{n-1} \left| a  \sin\left( \pi k/n \right) + b \cos\left(  \pi k/n\right) \right| + \mathcal{O}(n).$$
 The sum $\Sigma$ can be rewritten as
\begin{align*}
\Sigma &= \sqrt{a^2 + b^2} \sum_{k=0}^{n-1} \left| \frac{a}{\sqrt{a^2 + b^2}}  \sin\left( \pi k/n \right) + \frac{b}{\sqrt{a^2 + b^2}} \cos\left(  \pi k/n\right) \right| \\
&= \sqrt{a^2 + b^2} \sum_{k=0}^{n-1} \left|  \sin\left( \frac{\pi k}{n}  + \theta\right) \right|
\end{align*}
for a suitable $\theta$ depending only on $a,b$. This sum has been extensively studied, see Moran \cite{moran} or Steinhaus \cite{stein00}. Using the Fourier series
$$ |\sin{(x)}| = \frac{2}{\pi} - \frac{4}{\pi} \sum_{m=1}^{\infty} \frac{ \cos{(2mx)}}{4m^2-1}$$
and plugging in, we arrive at
\begin{align*}
\sum_{k=0}^{n-1} \left|  \sin\left( \frac{\pi k}{n}  + \theta\right) \right| &= \sum_{k=0}^{n-1} \left( \frac{2}{\pi} - \frac{4}{\pi} \sum_{m=1}^{\infty} \frac{ \cos{(2m \left(\frac{\pi k}{n}  + \theta \right))}}{4m^2-1} \right) \\
&= \frac{2n}{\pi} - \frac{4}{\pi}  \sum_{m=1}^{\infty} \frac{1}{4m^2-1}  \sum_{k=0}^{n-1} \cos{\left(2m \left(\frac{\pi k}{n}  + \theta \right)\right)}.
 \end{align*}
 The inner sum is 0 unless $m$ is a multiple of $n$. Thus
 $$ \sum_{k=0}^{n-1} \left|  \sin\left( \frac{\pi k}{n}  + \theta\right) \right| = \frac{2n}{\pi} - \frac{4}{\pi} \sum_{\ell=1}^{\infty} \frac{n}{4\ell^2 n^2-1}  \cos(2 \ell n \theta).$$
 One could now argue that this expression is minimized when $\theta = 0$ and maximized when $\theta = \pi/(2n)$ and obtain a precise closed-form expression in terms of the cotangent and the cosecant (which is done by Moran and Steinhaus). We do not require results at that level of precision and simply argue that
 $$ \left| \sum_{\ell=1}^{\infty} \frac{1}{4\ell^2 n^2-1} n \cos(2 \ell n \theta) \right| \leq \sum_{\ell=1}^{\infty} \frac{n}{4\ell^2 n^2-1} \leq \frac{1}{n}.$$
Altogether
\begin{align*}
\# \mbox{intersections} &=  \frac{1}{\varepsilon} \sum_{k=0}^{n-1} \left| a  \sin\left( \pi k/n \right) + b \cos\left(  \pi k/n\right) \right| + \mathcal{O}(n) \\
&=    \frac{1}{\varepsilon} \frac{2n}{\pi} \sqrt{a^2 + b^2}  + \mathcal{O}\left( \frac{\sqrt{a^2 + b^2}}{ \varepsilon n}\right) + \mathcal{O}(n)\\
&=    \frac{1}{\varepsilon} \frac{2n}{\pi} \sqrt{a^2 + b^2}  + \mathcal{O}\left( \frac{ \mbox{diam}(\Omega)}{ \varepsilon n} + n \right)
\end{align*} 
We see that the leading order term is directly proportional to the length $\sqrt{a^2 + b^2}$ which is exactly what we want.  It remains to obtain a uniform estimate for the error. We observe that the Steinhaus set $S_{n, \varepsilon}$ has length $\sim n/\varepsilon$ inside a unit square and thus also for any bounded domain with an implicit constant depending on $\Omega$.
 Then, fixing $n/\varepsilon \sim L$ to be the length of the set, we want to minimize
$$ \frac{1}{n \varepsilon} + n = \frac{L}{n^2} + n \qquad \mbox{leading to} ~n \sim L^{1/3} ~\mbox{and therefore}~\varepsilon \sim L^{2/3}.$$
This choice of parameters then yields the desired result.
\end{proof}

 \textbf{Acknowledgment.}  The author is grateful to Google DeepMind (special thanks to Bogdan Georgiev and Adam Zsolt Wagner) for facilitating the use of \textsc{AlphaEvolve}. It was used as an exploratory tool in the early stages of the manuscript where it was used to search for sets with good properties.

\end{document}